\documentclass[11pt,a4paper,twoside,english]{article}

\usepackage{amssymb,amsmath,amsthm}
\usepackage{graphicx}
\usepackage{babel}

\def\eps{\varepsilon}

\def\bs{\boldsymbol}
\def\hat{\widehat}

\def\Ex{\mathop{\mathrm{I\!E}}\nolimits}
\def\Pr{\mathop{\mathrm{I\!P}}\nolimits}

\newcommand{\R}{\mathbb{R}}

\newcommand{\beps}{\bs{\eps}}
\newcommand{\X}{\bs{X}}
\newcommand{\Y}{\bs{Y}}

\newcommand{\FF}{\mathcal{F}}
\newcommand{\XX}{\mathcal{X}}

\def\bea{\begin{eqnarray*}}
\def\eea{\end{eqnarray*}}

\newtheorem{Theorem}{Theorem}
\newtheorem{Lemma}{Lemma}

\begin{document}

\addtolength{\baselineskip}{0.3\baselineskip}

\begin{center}
	{\large University of Bern}\\
	{\large Institute of Mathematical Statistics and Actuarial Science}
	
	\bigskip
	\textbf{\large Technical Report 78}
\end{center}

\begin{center}
\textbf{\Large (Ab)Using Regression for Data Adjustment}

{\large Lutz D\"umbgen}\\
{\large September 2016}
\end{center}

\begin{abstract}
In various economic applications, people want to compare $n$ units with respect to certain quantities $Y_1, Y_2, \ldots, Y_n$ measuring their performance. The latter, however, is often influenced by certain factors which are beyond control of the units, and one would like to extract an adjusted performance from the data. Specifically, let $X_i \in \XX$ summarize the factors of the $i$-th unit. Then one could think of a model equation $Y_i = f_o(X_i) + \eps_i$ with a regression function $f_o : \XX \to \R$ describing the unavoidable influence of the factors $X_i$ and $\eps_i$ being the adjusted performance of the $i$-th unit. Now a common proposal is to estimate $f_o$ via regression methods by a function $\hat{f}$ depending on the current data $(X_i,Y_i)$, possibly augmented by additional past data, and to use the residuals $\hat{\eps}_i := Y_i - \hat{f}(X_i)$ as surrogates for the adjusted performances $\eps_i$. In the present report we discuss this approach, its potential pitfalls and (mis)interpretation. In particular, an unavoidable property of the residuals $\hat{\eps}_i$ is that they measure only parts of the adjusted performance while the remaining parts get hidden in the estimated function $\hat{f}$. Possible alternatives are mentioned briefly.
\end{abstract}

%=====================
\section{Introduction}
%=====================

This report is motivated by various consulting cases all of which involved a variant of the following method: Suppose that one wants to compare $n$ units with respect to certain quantities $Y_1, Y_2, \ldots, Y_n$ measuring their performance. Two specific examples are:

\begin{itemize}
\item The units are hospitals, and for a certain type of diseases, $Y_i$ measures the mean success of treatment within the $i$-th hospital.
\item The units are service areas of a big postal service, and $Y_i$ measures the mean delivery time for an item in service area $i$. (Here one aims for low values of $Y_i$.)
\end{itemize}

As usual in regression, we refer to the $Y_i$ as `responses', although in the present context `raw performance measures' would be appropriate, too. Typically the performance is influenced by certain factors which cannot be controlled by the units, and one would like a fair comparison, taking such differences into account. In the example of hospitals, some hospitals may tend to get more problematic cases than others or the mix of clients may vary from hospital to hospital with respect to important factors such as age, social background or initial diagnosis (in case of considering various diseases simultaneously). In the example of postal service areas, the population density is certainly negatively correlated with the delivery times. In general let $X_i \in \XX$ be a tuple containing the potentially relevant factors of the $i$-th unit. Then one could think of a model equation
\begin{equation}
\label{eq:additive.model}
	Y_i \ = \ f_o(X_i) + \eps_i
	\quad\text{for} \ 1 \le i \le n
\end{equation}
with a regression function $f_o : \XX \to \R$ describing the unavoidable influence of the tuple $X_i$, while $\eps_i$ is the adjusted performance (`true performance') of the $i$-th unit. One could also think of a multiplicative model with $Y_i = f_o(X_i) \cdot \eps_i$ with positive quantities $f_o(X_i)$ and $\eps_i$, but then a log-transformation would lead us to the additive equation \eqref{eq:additive.model} with $\log Y_i$, $\log f_o(X_i)$ and $\log \eps_i$ in place of $Y_i$, $f_o(X_i)$ and $\eps_i$, respectively.

If we were able to estimate the regression function $f_o$ with sufficient accuracy by a function $\hat{f}$ from the actual data $(X_1,Y_1), \ldots, (X_n,Y_n)$, possibly augmented by additional past data, one could use the residuals
\[
	\hat{\eps}_i \ = \ Y_i - \hat{f}(X_i)
\]
as surrogates for the adjusted performances $\eps_i$. A standard way to estimate $f_o$ would be via some regression method, for instance, least squares estimation of $f_o$ under the assumption that it belongs to a given family $\FF$ of regression functions $f : \XX \to \R$. Thus we estimate $f_o$ by a function $\hat{f} \in \FF$ such that
\[
	\sum_{i=1}^n (Y_i - \hat{f}(X_i))^2
\]
becomes minimal. In this report we focus on this approach with linear models $\FF$. That means, $\FF$ is assumed to be a finite-dimensional linear space of functions on $\XX$. The simplest case is multiple linear regression: Suppose that $x \in \XX$ stands for a tuple of $K$ numerical or $\{0,1\}$-valued covariables $x^{(1)}, \ldots, x^{(K)}$. Then one could consider functions
\[
	f(x) \ = \ a + \sum_{k=1}^K b_k^{} x^{(k)}
\]
with real constants $a, b_1, \ldots, b_K$. The set $\FF$ of all such functions $f$ is a linear space of dimension $K+1$. Alternatively one could consider multiple quadratic regression involving functions of the form
\[
	f(x) \ = \ a + \sum_{k=1}^K b_k^{} x^{(k)}
		+ \sum_{k=1}^K \sum_{\ell=k}^K g_{k,\ell}^{} x^{(k)} x^{(\ell)}
\]
with real constants $a$, $b_k$ and $g_{k,\ell}$, $1 \le k \le \ell \le K$. This corresponds to a function space $\FF$ of dimension $(K+1)(K+2)/2$. For a general introduction to regression methods we refer to Ryan~(1997) or D\"umbgen~(2015).

In Section~\ref{sec:Non-Identifiability} we discuss (non-)identifiability of $f_o$ in \eqref{eq:additive.model} and its implications. In particular we contrast the present (ab)use of regression with common applications and mention briefly an alternative approach called data envelopement analysis.

In Section~\ref{sec:Suppose} we discuss the potential impact of replacing the $\eps_i$ with the residuals $\hat{\eps}_i$ under the rather optimistic assumption that the $\eps_i$ may be viewed as realisations of independent Gaussian random variables with mean $0$ and a common standard deviation $\sigma > 0$. It is shown that ranking of units via $\hat{\eps}_i$ is strongly influenced by the observations' leverages.

Finally, in Section~\ref{sec:Two-Way} we describe a potential alternative method which is feasible whenever for each unit $i$ we have several observations (`cases')
\[
	(X_{ij},Y_{ij}) , \quad 1 \le j \le J_i .
\]
For instance, in the example of hospitals, we may have data of $J_i$ patients or treatments in hospital $i$, and the tuples $X_{ij}$ may be case-specific rather than hospital-specific. Then a possible alternative to model equation \eqref{eq:additive.model} is given by
\begin{equation}
\label{eq:additive.model.2}
	Y_{ij} \ = \ f_o(X_{ij}) + a_i + \epsilon_{ij}
\end{equation}
with a regression function $f_o : \XX \to \R$ as before, unit-specific parameters $a_i$ measuring their performances and random errors $\epsilon_{ij}$.

Some technical arguments and proofs are deferred to Section~\ref{sec:Proofs}.

%=================================================
\section{Non-identifiability and its implications}
\label{sec:Non-Identifiability}
%=================================================

For the moment let us view $\beps = (\eps_i)_{i=1}^n$ as a fixed $n$-dimensional vector. Without further assumptions on $\beps$, the function $f_o$ is not well-defined through \eqref{eq:additive.model}. We could replace $f_o$ and $\beps$ with $f_o + \Delta$ and $(\eps_i - \Delta(X_i))_{i=1}^n$, respectively, where $\Delta$ is an arbitrary function in our model space $\FF$, and \eqref{eq:additive.model} would remain true.

For $x \in \XX$ the value $f_o(x)$ could be interpreted as the average response over all units, if all of them were forced to work under conditions as specified by $x$. Thus a natural additional requirement seems to be that
\begin{equation}
\label{eq:centered.eps}
	\sum_{i=1}^n \eps_i \ = \ 0 .	
\end{equation}
But \eqref{eq:centered.eps} alone does not alleviate our identifiability problem. Suppose that $\FF$ contains all constant functions. For any non-constant function $\Delta \in \FF$ we could replace $f_o$ and $\beps$ with $f_o + \Delta - c$ and $(\eps_i - \Delta(X_i) + c)_{i=1}^n$, respectively, where $c := n^{-1} \sum_{i=1}^n \Delta(X_i)$. Then both \eqref{eq:additive.model} and \eqref{eq:centered.eps} would remain true.

Recall that the least squares estimator of $f_o$ is a function $\hat{f} \in \FF$ minimising the sum of squares $\sum_{i=1}^n (Y_i - \hat{f}(X_i))^2$. Geometrically speaking, the vector $\hat{f}(\X) = \bigl( \hat{f}(X_i) \bigr)_{i=1}^n$ is the orthogonal projection of $\Y = (Y_i)_{i=1}^n$ onto the linear subspace
\[
	\FF(\X) \ := \ \bigl\{ f(\X) : f \in \FF \bigr\}
\]
of $\R^n$, where $f(\X) := (f(X_i))_{i=1}^n \in \R^n$. In particular, $\hat{\beps} = (\hat{\eps}_i)_{i=1}^n$ is the orthogonal projection of $\bs{Y}$ onto the space $\FF(\X)^\perp$ of all vectors $\bs{v} \in \R^n$ which are perpendicular to $\FF(\X)$. If $f_o \in \FF$, then $\hat{\beps} = (\hat{\eps}_i)_{i=1}^n$ is the orthogonal projection of $\beps = (\eps_i)_{i=1}^n$ onto the space $\FF(\X)^\perp$. Then $\hat{f}(\X) = f_o(\X)$ and $\hat{\beps} = \beps$ if, and only if,
\begin{equation}
\label{eq:eps.perp.FF}
	\beps \ \perp \ \FF(\X) .
\end{equation}

\paragraph{Violations of \eqref{eq:eps.perp.FF}.}
This assumption is often hard to justify. One can easily imagine situations in which it is violated: In the setting of hospitals, suppose that each tuple $X_i$ includes a quantity $X_i^{(1)}$ measuring average severity of cases treated in hospital $i$. Hospitals which tend to treat the more difficult cases could also hire particularly experienced or highly qualified personnel, despite the higher costs. Presumably this would result in a positive sample correlation for the pairs $(X_i^{(1)}, \eps_i)$, $1 \le i \le n$, a clear violation of \eqref{eq:eps.perp.FF} in the setting of multiple linear or quadratic regression. In the setting of postal service areas, suppose that $X_i$ includes a measure $X_i^{(1)}$ of population density. It may happen that service areas in low density regions take extra efforts to accelerate deliveries, trying to alleviate this unfavourable condition. This would result in a positive sample correlation for the pairs $(X_i^{(1)}, \eps_i)$ if $Y_i$ stands for average delivery time, again a violation of \eqref{eq:eps.perp.FF}. But in both settings the positive effect of the extra efforts would be eliminated by the regression method in that it appears in $\hat{f}(\X)$ rather than in $\hat{\beps}$.

Alternatively, suppose that the first component $x^{(1)}$ of a tuple $x \in \XX$ is a categorical variable with values in $\{1,\ldots,L\}$. Thus the $n$ units may be divided into $L$ different categories. If this factor $x^{(1)}$ is considered to have a potential impact on the units' performance, one should require $\FF$ to contain at least all functions
\[
	f(x) \ = \ \sum_{\ell=1}^L a_\ell^{} 1_{[x^{(1)} = \ell]}
\]
with real constants $a_1,\ldots,a_L$. But then the residuals will automatically satisfy the equations
\[
	\sum_{i \,:\, X_i^{(1)} = \ell} \hat{\eps}_i \ = \ 0
	\quad\text{for} \ 1 \le \ell \le L .
\]
It may happen that the units in a particular category $\ell_o$ perform above or below average, meaning that $\sum_{i \,:\, X_i^{(1)} = \ell_o} \eps_i$ is strictly positive or negative, on top of potential differences between the sums $\sum_{i \,:\, X_i^{(1)} = \ell} f_o(X_i)$, $1 \le \ell \le L$, which one intends to adjust for. But the regression residuals won't reflect such patterns.

\paragraph{A simple example.}
To illustrate these potential problems, we simulated $n = 25$ observation pairs $(X_i,Y_i)$ with univariate $X_i$ such that $f_o(x) = 1 - x/2$. The sample correlation between $\bs{X} = (X_i)_{i=1}^n$ and $\Y$ equals $-0.981$, but $\beps$ and $\bs{X} = (X_i)_{i=1}^n$ have sample correlation $0.831$. The left panel of Figure~\ref{fig:Demo1ab} shows the data pairs $(X_i,Y_i)$, the function $f_o$ (thick line) as well as the fitted function $\hat{f}(x) = \hat{a} + \hat{b} x$ (dashed line). Here $\FF$ was taken to be the set of all affine functions $f(x) = a + bx$. Due to the negative sample correlation between $\X$ and $\beps$, the estimated function $\hat{f}$ with slope $\hat{b} = -0.385$ is less steep then the true regression function $f_o$ with slope $-0.5$.

In the right panel we see the adjusted performances $\eps_i$ and their estimators, i.e.\ the residuals $\hat{\eps}_i$. For $20$ observations, the signs of $\eps_i$ and $\hat{\eps}_i$ were identical, for the other $5$ observations they differed! In Figure~\ref{fig:Demo1c} one sees the pairs $(R_i,\hat{R}_i)$, where $R_i$ is the rank of $\eps_i$ within $\beps$ while $\hat{R}_i$ is either the rank of $\hat{\eps}_i$ within $\hat{\beps}$ (bullets) or of $Y_i$ within $\bs{Y}$ (circles). Obviously, the differences $\hat{R}_i - R_i$ may be substantial. Replacing the raw values $Y_i$ with the residuals $\hat{\eps}_i$ results in an improvement for $18$ and no change for $3$ observations; but for $4$ observations the estimated rank gets worse.

\begin{figure}
\includegraphics[width=0.49\textwidth]{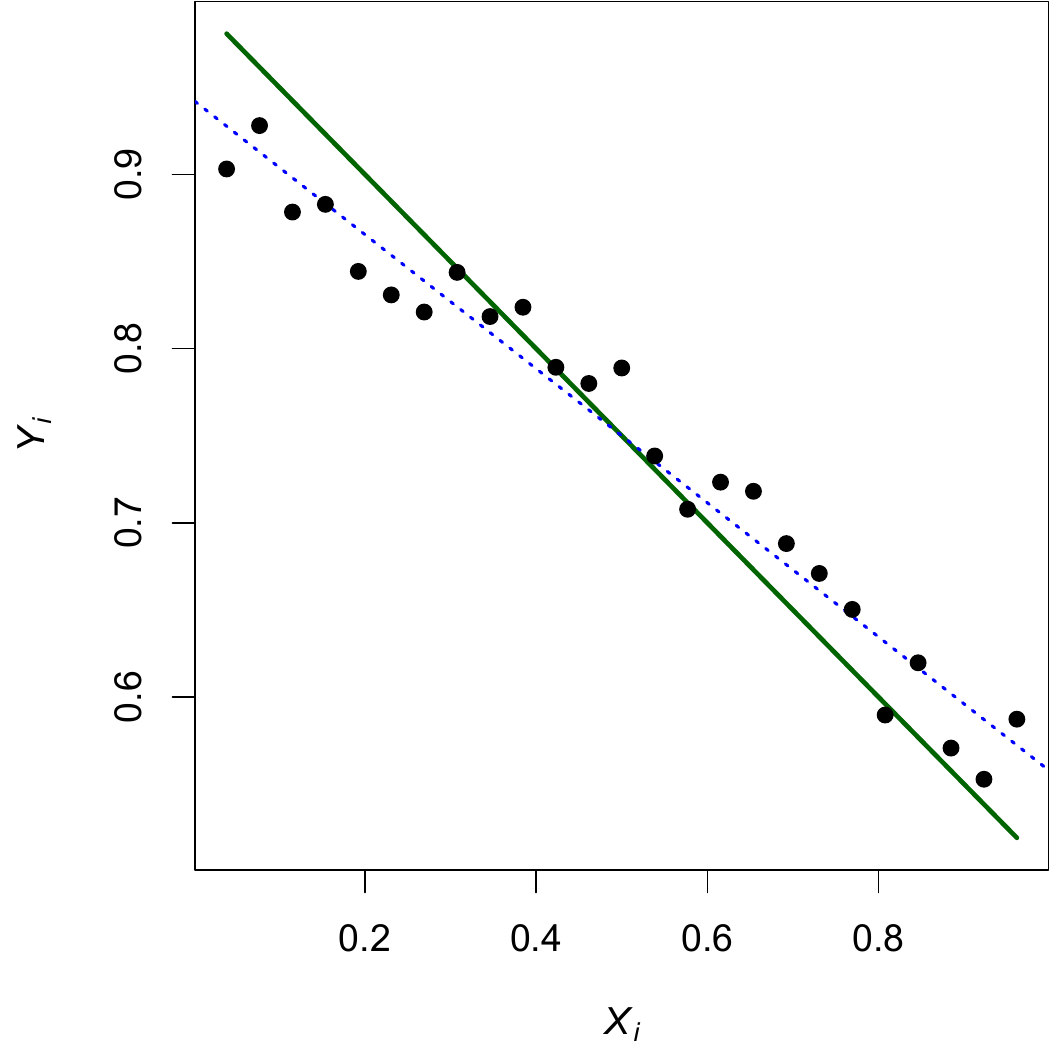}
\hfill
\includegraphics[width=0.49\textwidth]{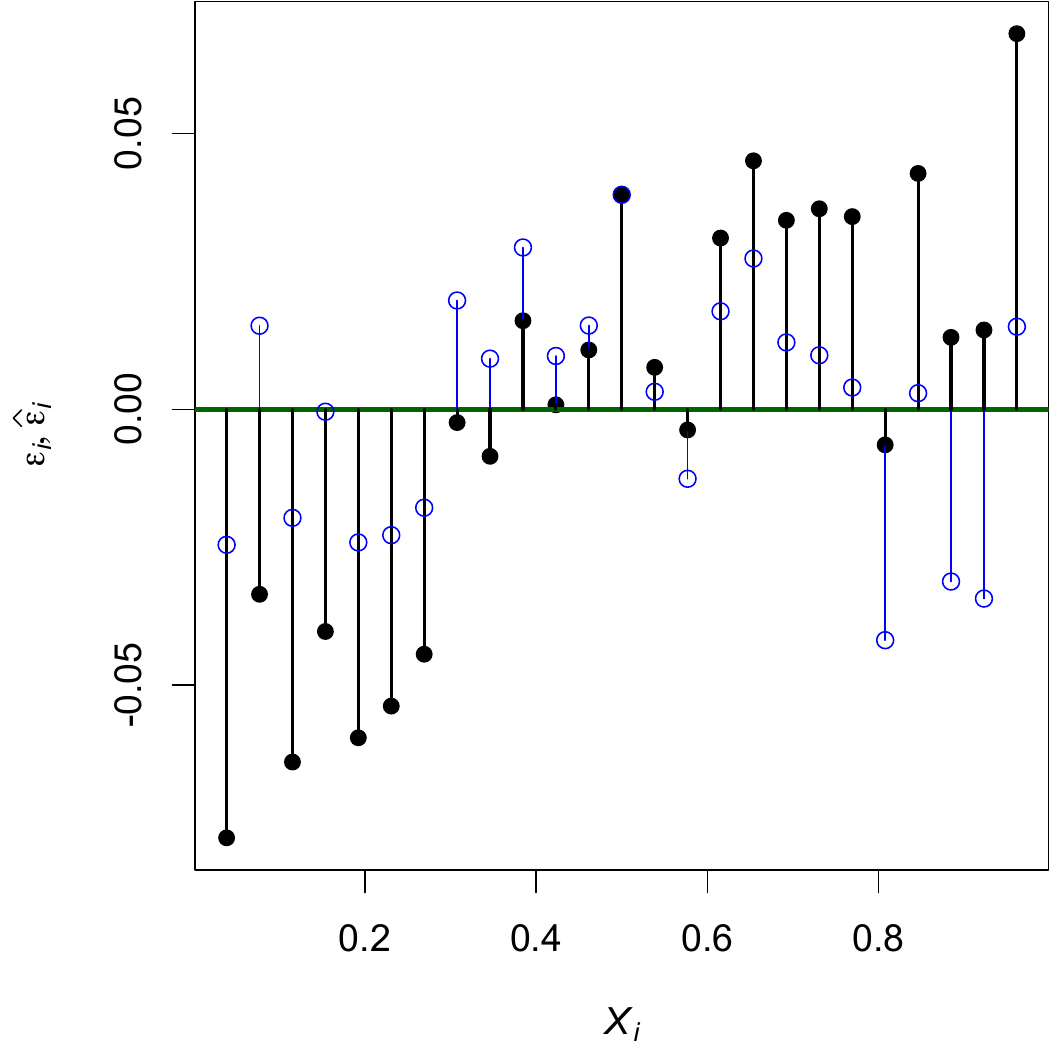}

\caption{Simple example: Left panel: Raw data $(X_i,Y_i)$ with $f_o$ (line) and $\hat{f}$ (dashed line). Right panel: Adjusted performances $\eps_i$ (bullets) and residuals $\hat{\eps}_i$ (circles).}
\label{fig:Demo1ab}
\end{figure}

\begin{figure}
\centering
\includegraphics[width=0.60\textwidth]{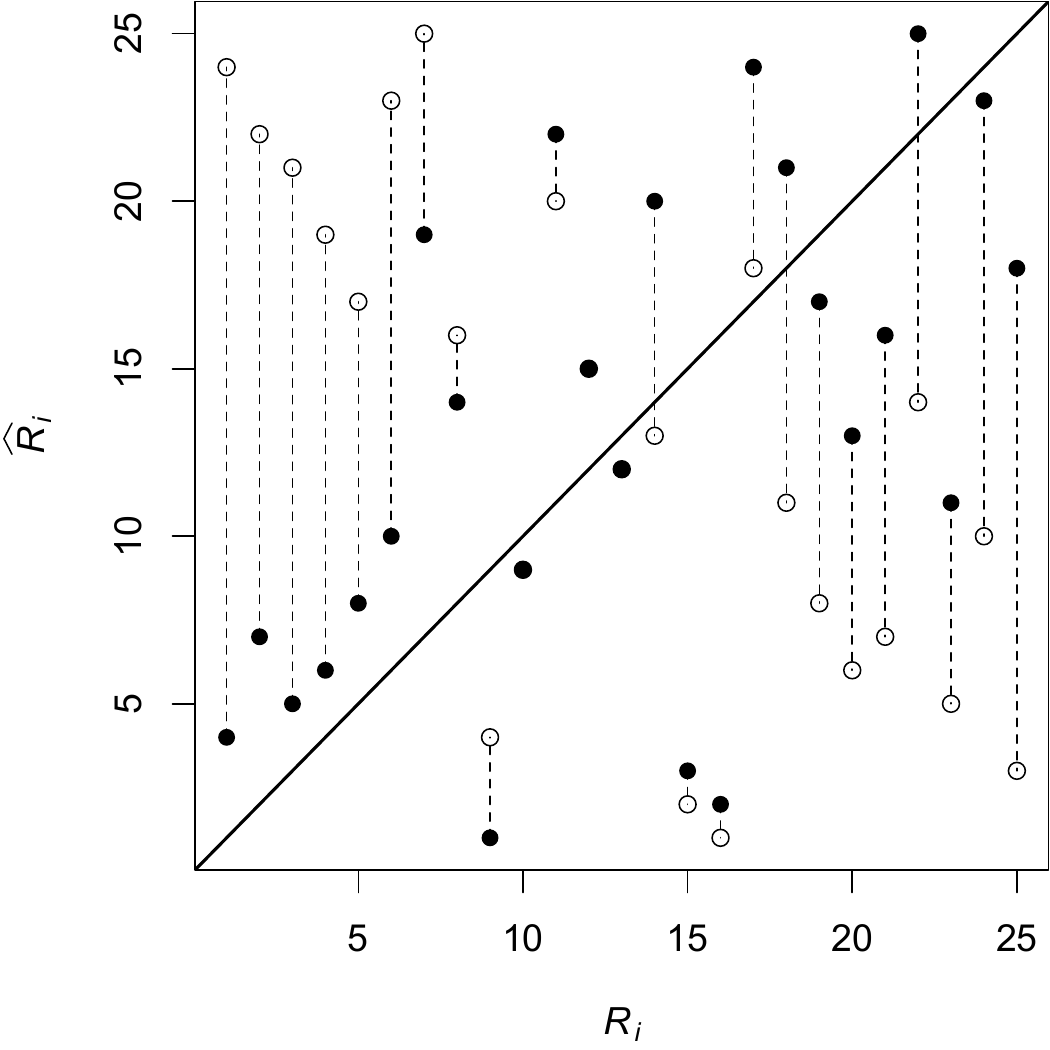}

\caption{Simple example: Ranks $R_i$ of the adjusted performances $\eps_i$ and estimated ranks $\hat{R}_i$ based on residuals $\hat{\eps}_i$ (bullets) or raw data $Y_i$ (circles).}
\label{fig:Demo1c}
\end{figure}

\paragraph{Partial performance.}
The previous considerations and simple example illustrate the dilemma we have to deal with: If the sample correlation between $f_o(\X)$ and $\Y$ is substantial, it is better to estimate $f_o$ than to ignore it. But then $\hat{\beps}$ measures only \textsl{partial performance}. That means, roughly saying, the vector $\hat{\beps}$ contains only those aspects of the adjusted performances which are not correlated with $\X$. Precisely, $\hat{\beps}$ is the orthogonal projection of $\beps$ onto the linear space $\FF(\X)^\perp$.

Whenever one uses $\hat{\beps}$ to quantify performance of the $n$ units, one should keep in mind that substantial parts of the adjusted performances $\eps_i$ may be missing, and even the sign of $\hat{\eps}_i$ may differ from the sign of $\eps_i$. This is particularly important whenever analyses of $\hat{\beps}$ are used to create incentives for better performance. If the different units are figuring out ways to improve performance, it may happen that self-evident steps are somehow related to $\X$. In that case, even if these measures are successful, the resulting improvements may disappear in the regression adjustment. Somehow this gives a new twist to the well-known phenomenon of `regression towards mediocrity'.

\paragraph{Further `leakage' due to misspecified models.}
Even in case of \eqref{eq:eps.perp.FF}, the estimator $\hat{\beps}$ may contain systematic errors due to model-misspecification. In general,
\[
	\Y \ = \ \bs{\mu} + \beps
\]
with $\bs{\mu} = f_o(\X)$ if the $X_i \in \XX$ contain all relevant covariables. Even then it may happen that $f_o \not\in \FF$. If we missed some important covariables $X_i' \in \XX'$, then $\bs{\mu} = \bigl( f_*(X_i,X_i') \bigr)_{i=1}^n$ with some regression function $f_* : \XX \times \XX' \to \R$. In both cases it is likely that $\bs{\mu} \not\in \FF(\X)$, and
\begin{align*}
	\hat{\beps} \
	&= \ (\bs{I} - \bs{H}) \bs{\mu} + (\bs{I} - \bs{H}) \beps \\
	&= \ (\bs{I} - \bs{H}) \bs{\mu}	+ \beps \qquad \text{under} \ \eqref{eq:eps.perp.FF} .
\end{align*}
Here $\bs{I} \in \R^{n \times n}$ denotes the identity matrix, and $\bs{H} \in \R^{n\times n}$ describes the orthogonal projection onto $\FF(\X)$, see also Section~\ref{sec:Suppose}.

\paragraph{Performances and random errors.}
Note that we treated the vector $\beps$ as fixed rather than random in the preceding considerations. Now one could think of
\[
	\eps_i \ = \ \pi_i + \delta_i
\]
with adjusted performances $\pi_i$ and truly random errors $\delta_i$ representing random fluctuations which are neither related to the $X_i$ nor to the units' abilities or failures. Without more advanced sampling schemes and certain assumptions on the $\pi_i$, however, there is no possibility to estimate the latter quantities.

In traditional regression applications, people are mainly interested in $f_o$ as a means to describe the relation between $X_i$ and $Y_i$ and to predict $Y_o$ from $X_o$ for future observations $(X_o,Y_o)$ which are indpendent of the given pairs $(X_i,Y_i)$, $1 \le i \le n$. Here the $\eps_i$ represent measurement or sampling errors which are modelled as random and considered to be an unavoidable nuisance. The residuals $\hat{\eps}_i$ are only used to estimate certain properties of the random errors $\eps_i$ and to validate or falsify certain assumptions on these. Using regression as described previously to create ``corrected values'' $\hat{\eps}_i$ as a surrogate for the raw values $Y_i$ is not what the method is designed for; therefore the provocative word `abusing' in the title of this report.

\paragraph{Data envelopement analysis and quantile regression.}
One should mention here an established method of benchmarking, called data envelopement analysis (DEA), initiated by Farrel (1957) and Charnes et al.\ (1978). Very roughly saying, in that approach one assumes that the $\eps_i$ are non-positive (if higher values of $Y_i$ mean better performance), and $f_o(x)$ is the maximally achievable performance under the circumstances described by $x \in \XX$. The deviations $\eps_i$ are then estimated via a linear optimization method. The main reasons for using a regression approach rather than DEA seem to be the higher complexity of DEA, which makes it more difficult to communicate it to laymen, and the known sensitivity of DEA to errors in the data. Moreover, normal quantile-quantile plots of the residuals $\hat{\eps}_i$ often indicate a Gaussian distribution, whereas the DEA paradigm would predict a non-symmetric, left-skewed distribution.

If the residuals show indeed a left-skewed distribution, a possible compromise between least squares regression and DEA would be regression quantiles (cf.\ Koenker and Basset, 1978). That means, for a given parameter $\gamma \in (0,1)$ one determines a function $\hat{f}_\gamma$ minimising
\[
	\sum_{i=1}^n \rho_\gamma(Y_i - f(X_i))
\]
over all $f \in \FF$, where
\[
	\rho_\gamma(t)
	\ := \ \bigl( 1_{[t \ge 0]} \gamma + 1_{[t \le 0]} (1 - \gamma) \bigr) |t|
\]
for $t \in \R$. An advantage of this approach would be that the estimator $\hat{f}_\gamma$ is less sensitive to outliers in the $Y_i$ than the least squares estimator $\hat{f}$. For $\gamma$ close to one, the estimated function $\hat{f}_\gamma$ would be close to the envelope function $f_o$.

%===============================================
\section{Ranking errors in a best case scenario}
\label{sec:Suppose}
%===============================================

Let us set aside all reservations towards estimating $f_o$ by the least squares estimator $\hat{f}$. Rather than considering $\beps$ to be a fixed vector satisfying \eqref{eq:eps.perp.FF}, let us assume that its components $\eps_i$ are independent and identically distributed random variables with centered Gaussian distribution,
\[
	\eps_i \ \sim \ \mathcal{N}(0,\sigma^2)
\]
with unknown standard deviation $\sigma > 0$. Under the additional assumption that $f_o$ is contained in our model $\FF$, the estimator $\hat{f}$ and the residuals $\hat{\eps}_i$ are unbiased in the sense that
\[
	\Ex \hat{f}(x) \ = \ f_o(x)
	\quad\text{and}\quad
	\Ex (\hat{\eps}_i - \eps_i) \ = \ 0 .
\]

\paragraph{Leverage.}
Using the $\hat{\eps}_i$ as surrogates for the $\eps_i$, however, is problematic because of the well-known phenomenon of \textsl{leverage}: Observations $(X_i,Y_i)$ with ``exotic'' part $X_i$ tend to produce residuals $\hat{\eps}_i$ with smaller modulus than the adjusted performances $\eps_i$. Precisely, let $\bs{H} \in \R^{n\times n}$ be the so-called hat matrix for the given model $\FF$. That means, if $f_1, f_2, \ldots, f_p$ are basis functions of $\FF$, and if the corresponding design matrix
\[
	\bs{D} \ = \ \bigl[ f_1(\X), f_2(\X), \ldots, f_p(\X) \bigr]
%	\ = \ \begin{bmatrix}
%		f_1(X_1) & f_2(X_1) & \ldots & f_p(X_1) \\
%		f_1(X_2) & f_2(X_2) & \ldots & f_p(X_2) \\
%		\vdots        & \vdots        &        & \vdots        \\
%		f_1(X_n) & f_2(X_n) & \ldots & f_p(X_n)
%	\end{bmatrix}
	\ \in \ \R^{n\times p}
\]
has full rank $p < n$, then
\[
	\bs{H} \ = \ \bs{D} (\bs{D}^\top \bs{D})^{-1} \bs{D}^\top
\]
describes the orthogonal projection from $\R^n$ onto the $p$-dimensional model space $\FF(\X)$. The matrix $\bs{H}$ does not depend on the particular choice of basis functions $f_1, \ldots, f_p$ but on the tuple $\X$ of covariable vectors $X_i$ and the model $\FF$. Now
\[
	\hat{\beps} \ = \ (\bs{I} - \bs{H}) \beps .
\]
Elementary calculations and the fact that $\bs{H}^\top = \bs{H} = \bs{H}^2$ show that
\[
	\Ex(\hat{\eps}_i^2) \ = \ (1 - H_{ii}) \sigma^2
	\quad\text{and}\quad
	\mathrm{Corr}(\hat{\eps}_i, \eps_i) \ = \ \sqrt{1 - H_{ii}} .
\]
Thus the modulus of $\hat{\eps}_i$ may be systematically too small and the correlation of $\hat{\eps}_i$ and $\eps_i$ substantially smaller than $1$ in case of high values of the leverage $H_{ii} \in [0,1]$ of the $i$-th observation part $X_i$. Note also that
\[
	\sum_{i=1}^p H_{ii} \ = \ p
\]
whence
\[
	\max_{i=1,2,\ldots,n} H_{ii} \ \ge \ \frac{p}{n} .
\]
Thus a rather complex model with high dimension $p$ relative to the sample size will automatically yield observations with high leverages $H_{ii}$. But even in case of $p \ll n$ some leverages may be substantial.

\paragraph{Ranking errors.}
Suppose one uses the estimated performances $\hat{\eps}_i$ to rank the units. That means, one computes the rank $\hat{R}_i$ of $\hat{\eps}_i$ within $\hat{\beps}$ as a proxy for the rank $R_i$ of $\eps_i$ within $\beps$,
\[
	R_i
	\ = \ \sum_{j=1}^n 1_{[\eps_j \le \eps_i]} .
\]
In what follows we derive explicit expressions for the root mean squared ranking errors,
\[
	\sqrt{ \Ex \bigl( (\hat{R}_i - R_i)^2 \bigr) } .
\]

First of all, under a mild regularity condition on the hat matrix $\bs{H}$, the residuals $\hat{\eps}_i$ are pairwise different:

\begin{Lemma}
\label{Lemma:hat matrix}
For arbitrary indices $1 \le i < j \le n$,
\[
	\Pr \bigl( \hat{\eps}_i = \hat{\eps}_j \bigr)
	\ = \ \begin{cases}
		1 & \text{if} \ H_{ii} = H_{jj} = H_{ij} + 1 , \\
		0 & \text{else} .
	\end{cases}
\]
The condition $H_{ii} = H_{jj} = H_{ij} + 1$ implies that $H_{ii} \ge 1/2$.
\end{Lemma}

This lemma remains valid if the errors $\eps_1, \eps_2, \ldots, \eps_n$ are only assumed to be independent with continuous distributions. An immediate consequence of Lemma~\ref{Lemma:hat matrix} is that the residuals $\hat{\eps}_i$ are pairwise different almost surely, whenever $H_{ii} \ge 1/2$ for at most one index $i$. In particular, with probability one,
\[
	\hat{R}_i \ = \ \sum_{j=1}^n 1_{[\hat{\eps}_j \le \hat{\eps}_i]} .
\]

Here is a first main result about the ranks $R_i$ and $\hat{R}_i$:

\begin{Theorem}
\label{Theorem:exact}
Suppose that $H_{ii} \ge 1/2$ for at most one index $i \in \{1,2,\ldots,n\}$. Then for arbitrary indices $i,j \in \{1,2,\ldots,n\}$,
\begin{align*}
	\Ex \bigl( (\hat{R}_i - R_i) (\hat{R}_j - R_j) & \bigr) \ = \\
	\frac{1}{2\pi} \sum_{k,\ell=1}^n \biggl(
		& \arcsin \Bigl( \frac{\Delta_{k\ell,ij}}{2} \Bigr)
		+ \arcsin \Bigl( \frac{\Delta_{k\ell,ij} - H_{k\ell,ij}}
			{\sqrt{(2 - H_{kk,ii})(2 - H_{\ell\ell,jj})}} \Bigr) \\
		& - \ \arcsin \Bigl( \frac{\Delta_{k\ell,ij} - H_{k\ell,ij}}
			{\sqrt{2 (2 - H_{kk,ii})}} \Bigr)
		- \arcsin \Bigl( \frac{\Delta_{k\ell,ij} - H_{k\ell,ij}}
			{\sqrt{2 (2 - H_{\ell\ell,jj})}} \Bigr) \biggr) ,
\end{align*}
where $\Delta_{k\ell,ij} := \delta_{k\ell} + \delta_{ij} - \delta_{kj} - \delta_{i\ell}$ and $H_{k\ell,ij} := H_{k\ell} + H_{ij} - H_{kj} - H_{i\ell}$. In particular,
\begin{align*}
	\Ex \bigl( (\hat{R}_i - R_i)^2 \bigr)
		\ = \
		\frac{1}{2\pi} \sum_{k=1}^n
		\biggl( & \pi - 2 \arccos \Bigl( \sqrt{H_{kk,ii}/2} \Bigr) \biggr) \ + \\
	\frac{1}{\pi} \sum_{1 \le k < \ell \le n}
		\biggl( & \frac{\pi}{6} + \arcsin \Bigl( \frac{1 - H_{k\ell,ii}}
				{\sqrt{(2 - H_{kk,ii})(2 - H_{\ell\ell,ii})}} \Bigr) \\
	& - \ \arcsin \Bigl( \frac{1 - H_{k\ell,ii}}{\sqrt{2 (2 - H_{kk,ii})}} \Bigr)
			- \arcsin \Bigl( \frac{1 - H_{k\ell,ii}}{\sqrt{2 (2 - H_{\ell\ell,ii})}}
			\Bigr) \biggr) .
\end{align*}
\end{Theorem}

Here and throughout $\delta_{st}$ denotes Kronecker's symbol, $\delta_{st} = 1_{[s = t]}$. Theorem~\ref{Theorem:exact} is useful for exact numerical calculations. Numerical experiments reveal also that the rank distortions are closely related to the leverages $H_{ii}$. Here is a theoretical result about the rank distortions in case of small maximal leverage:

\begin{Theorem}
\label{Theorem:approximate}
Suppose that the column space of $\bs{D}$ contains the constant vectors, i.e.\ $\bs{H} \bs{1} = \bs{1} := (1)_{i=1}^n$. Then, as $\eta := \max_{i=1,\ldots,n} H_{ii} \to 0$,
\[
	\Ex \bigl( (\hat{R}_i - R_i)(\hat{R}_j - R_j) \bigr)
	\ = \ \frac{n^2 H_{ij} - n}{2 \sqrt{4 - \delta_{ij}} \, \pi}
			+ O \bigl( n \eta^{1/2} + n^2 \eta^2 \bigr)
\]
uniformly in $i,j \in \{1,2,\ldots,n\}$.
\end{Theorem}

\paragraph{A numerical example.}
Suppose that $n = 70$, and let $\bs{D}$ be equal to
\[
	\bs{D}^{(1)} \ := \ \begin{bmatrix}
		1 & X_1 \\
		1 & X_2 \\
		\vdots & \vdots \\
		1 & X_n
	\end{bmatrix}
	\quad\text{or}\quad
	\bs{D}^{(2)} \ := \ \begin{bmatrix}
		1 & X_1 & X_1^2\\
		1 & X_2 & X_2^2\\
		\vdots & \vdots & \vdots\\
		1 & X_n & X_n^2
	\end{bmatrix} ,
\]
the design matrix for simple linear or quadratic regression, where $X_1 < X_2 < \cdots < X_n$ are equispaced numbers. The maximal leverage $\max_i H_{ii}$ is equal to $0.0559$ for $\bs{D}^{(1)}$ and $0.1215$ for $\bs{D}^{(2)}$. Figure~\ref{fig:Example} shows the root mean squared ranking errors $\sqrt{\Ex \bigl( (\hat{R}_i - R_i)^2 \bigr)}$ for both cases. In addition the approximations $n \sqrt{(H_{ii} - n^{-1})/(2 \pi \sqrt{3})}$ and $n \sqrt{H_{ii}/(2 \pi \sqrt{3})}$ are shown as lines.

\begin{figure}[h]
\centering
\includegraphics[width=1.0\textwidth]{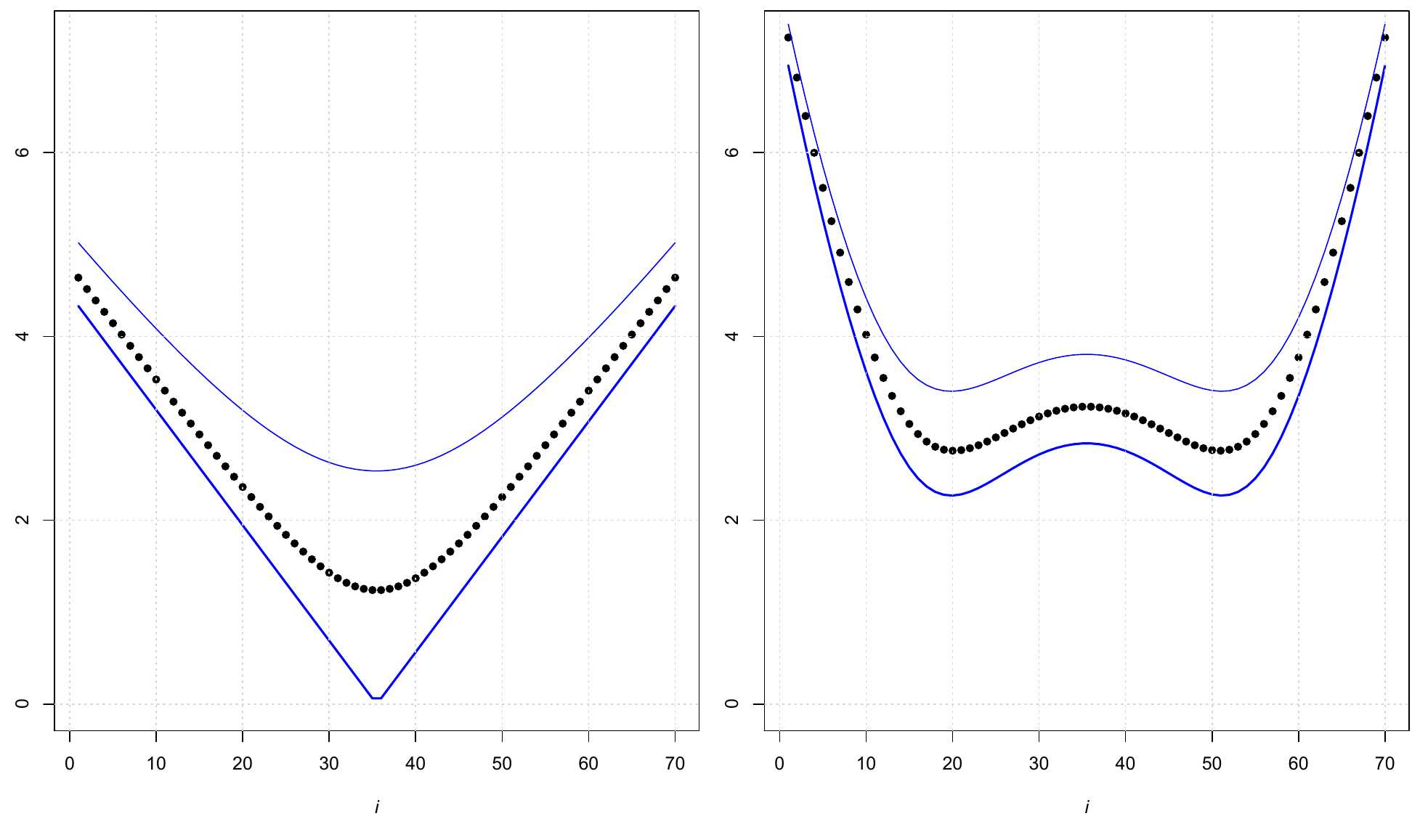}
\caption{Root mean squared ranking errors for simple linear regression (left) and quadratic regression (right) with $n=70$ equispaced $X$-values.}
\label{fig:Example}
\end{figure}

\paragraph{Heuristics.}
Now we present heuristic arguments to approximate the rank distortions which also indicates what may happen in non-Gaussian settings. Presumably these arguments could be made rigorous by applying similar techniques and arguments as Koul (1969, 1992), Loynes (1980) and Mammen (1996). As in Theorem~\ref{Theorem:approximate}, asymptotic statements are meant as $\eta = \max_{i=1,\ldots,n} H_{ii} \to 0$. We assume that the errors $\eps_1, \ldots, \eps_n$ are independent and identically distributed with finite standard deviation $\sigma$ and distribution function $F_\eps$ with bounded and uniformly continuous density $f_\eps$.

One can easily deduce from $\bs{H}^\top = \bs{H} = \bs{H}^2$ that
\begin{equation}
\label{eq:Hij}
	\Ex \bigl( (\bs{H}\bs{\eps})_i (\bs{H}\bs{\eps})_j \bigr)
	\ = \ \sigma^2 H_{ij} ,
\end{equation}
whereas the Cauchy-Schwarz inequality implies that
\[
	\bigl| \Ex \bigl( (\bs{H}\bs{\eps})_i (\bs{H}\bs{\eps})_j \bigr) \bigr|
	\ \le \ \mathrm{Std}((\bs{H}\bs{\eps})_i) \mathrm{Std}((\bs{H}\bs{\eps})_j)
	\ = \ \sigma^2 \sqrt{H_{ii} H_{jj}} .
\]
Hence
\begin{equation}
\label{eq:H and leverage}
	|H_{ij}| \ \le \ \eta
	\quad\text{for} \ 1 \le i,j \le n .
\end{equation}
 
Pretending that the empirical c.d.f.\ $\check{F}$ of the errors $\eps_i$ and the empirical c.d.f.\ $\hat{F}$ of the residuals $\hat{\eps}_i$ are sufficiently close to $F$, we write
\begin{align*}
	R_i \ = \ n \check{F}(\eps_i)
		\ &\approx \ n F_\eps(\eps_i) , \\
	\hat{R}_i \ = \ n \hat{F}(\hat{\eps}_i)
		\ &\approx \ n F_\eps(\hat{\eps}_i)
			\ = \ n F_\eps \bigl( \eps_i - (\bs{H}\bs{\eps})_i \bigr) .
\end{align*}
But $(\bs{H}\bs{\eps})_i$ is quite small, precisely,
\[
	\Ex \bigl( (\bs{H}\bs{\eps})_i^2 \bigr)
	\ = \ \sigma^2 H_{ii} \ \le \ \sigma^2 \eta
\]
by \eqref{eq:Hij}. Hence we write
\[
	\hat{R}_i - R_i
	\ \approx \ - n f_\eps(\eps_i) (\bs{H}\bs{\eps})_i .
\]
Moreover, for $i,j \in \{1,2,\ldots,n\}$ and $\ell \in \{i,j\}$,
\[
	(\bs{H}\bs{\eps})_\ell
	\ = \ \sum_{k=1}^n H_{\ell k} \eps_k
	\ \approx \ \sum_{k \not\in \{i,j\}} H_{\ell k} \eps_k ,
\]
because $\sum_{k \in \{i,j\}} H_{\ell k}\eps_k$ is very small in the sense that
\[
	\Ex \biggl( \Bigl( \sum_{k \in \{i,j\}} H_{\ell k}\eps_k \Bigr)^2 \biggr)
	\ = \ \sigma^2 \sum_{k \in \{i,j\}} H_{\ell k}^2
	\ \le \ 2 \sigma^2 \eta^2
\]
by \eqref{eq:H and leverage}. Thus we pretend that the random pairs $(\eps_i, \eps_j)$ and $\bigl( (\bs{H}\bs{\eps})_i, (\bs{H}\bs{\eps})_j \bigr)$ are stochastically independent and conjecture that
\begin{align}
	\nonumber
	\Ex \bigl( (\hat{R}_i - R_i)(\hat{R}_j - R_j) \bigr) \
	&\approx \ n^2 \, \Ex \bigl( f_\eps(\eps_i) f_\eps(\eps_j) (\bs{H}\bs{\eps})_i (\bs{H}\bs{\eps})_j \bigr) \\
	\nonumber
	&\approx \ n^2 \, \Ex \bigl( f_\eps(\eps_i) f_\eps(\eps_j) \bigr) 
		\Ex \bigl( (\bs{H}\bs{\eps})_i (\bs{H}\bs{\eps})_j \bigr) \\
	\label{eq:approximation}
	&= \ n^2 \sigma^2 \Ex \bigl( f_\eps(\eps_i) f_\eps(\eps_j) \bigr) H_{ij} .
\end{align}

Now consider the special case of $F = \Phi(\sigma^{-1} \cdot)$ and $f = \sigma^{-1} \phi(\sigma^{-1} \cdot)$ with the standard Gaussian c.d.f.\ $\Phi$ and density $\phi$. For $i \ne j$,
\begin{align*}
	\sigma^2 \Ex \bigl( f_\eps(\eps_i) f_\eps(\eps_j) \bigr) \
	&= \ \bigl( \sigma \Ex f_\eps(\eps_i) \bigr)^2 \\
	&= \ \biggl( \sigma^{-1} \int \phi(\sigma^{-1} x)^2 \, dx \biggr)^2
		\ = \ \biggl( (2\pi)^{-1/2} \int \phi(\sqrt{2} x) \, dx \biggr)^2 \\
	&= \ (4\pi)^{-1} ,
\intertext{and}
	\sigma^2 \Ex \bigl( f_\eps(\eps_i)^2 \bigr) \
	&= \ \sigma^{-1} \int \phi(\sigma^{-1} x)^3 \, dx
		\ = \ (2\pi)^{-1} \int \phi(\sqrt{3} x) \, dx \\
	&= \ (2 \sqrt{3} \, \pi)^{-1} .
\end{align*}
Hence the conjectured approximation \eqref{eq:approximation} equals
\[
	\frac{n^2 H_{ij}}{2 \sqrt{4 - \delta_{ij}} \, \pi} .
\]

%==================================================
\section{Alternative methods for case-by-case data}
\label{sec:Two-Way}
%==================================================

As indicated in the introduction, suppose that for each unit $i$ we have several observations (`cases')
\[
	(X_{ij},Y_{ij}) , \quad 1 \le j \le J_i .
\]
For instance, in the example of hospitals, we may have data of $J_i$ patients or treatments in hospital $i$, and the tuples $X_{ij}$ may be case-specific rather than hospital-specific. In the example of service areas of a postal service, cases could be items to be delivered with the tuple $X_{ij}$ describing the type, size and weight of the item and characteristics of the receiver or the neighborhood he or she is living in. A potential problem, though, would be the determination of the single delivery times $Y_{ij}$.

Then a possible alternative to model equation \eqref{eq:additive.model} is given by \eqref{eq:additive.model.2}, i.e.
\[
	Y_{ij} \ = \ f_o(X_{ij}) + a_i + \epsilon_{ij}
\]
with a regression function $f_o : \XX \to \R$ as before, unit-specific parameters $a_i$ measuring their performances and random errors $\epsilon_{ij}$ with mean zero. Identifiability can be achieved by requiring that
\[
	\sum_{i=1}^n \sum_{j=1}^{J_i} f_o(X_{ij}) \ = \ 0 .
\]
Then the parameter $a_i$ would be the expected performance of unit $i$ if it was dealing with a randomly chosen case from all $J_+ := \sum_{i=1}^n J_i$ cases.

For instance, if any $x \in \XX$ is a tuple $(x^{(k)})_{k=1}^K$ of $K$ numerical or $\{0,1\}$-valued covariables, one could think about multiple linear regression:
\[
	f_o(x) \ := \ \sum_{k=1}^K \beta_k (x^{(k)} - \bar{X}^{(k)})
\]
with the overall averages
\[
	\bar{X}^{(k)} \ := \frac{1}{J_+} \sum_{i=1}^n \sum_{j=1}^{J_i} X_{ij}^{(k)}
\]
and certain real parameters $\beta_k$. Alternatively one could consider multiple quadratic regression:
\[
	f_o(x) \ := \ \sum_{k=1}^K \beta_k (x^{(k)} - \bar{X}^{(k)})
		+ \sum_{k=1}^K \sum_{\ell=k}^K \beta_{k,\ell}
			(x^{(k)} x^{(\ell)} - \bar{X}^{(k,\ell)})
\]
with the average $\bar{X}^{(k,\ell)}$ of all $J_+$ products $X_{ij}^{(k)} X_{ij}^{(\ell)}$ and certain real parameters $\beta_k, \beta_{k,\ell}$.

An important restriction is that we do not allow for interactions between units and the $K$ covariables. That means, the impact of the $K$ covariables is the same for all $n$ units. Without such an assumption we would encounter similar identifiability problems as in the simpler regression setting discussed before.

Now the $\eps_{ij}$ are really considered as random errors, and the performance parameters $a_i$ may be estimated via least squares estimators $\hat{a}_i$. In addition, one could determine standard errors for these estimators and single or simultaneous confidence bounds for the underlying parameters $a_i$. Depending on standard residual diagnostics, the latter could be based on a generalisation of Tukey's method for linear models with homoscedastic Gaussian errors. Alternatively, if homoscedasticity is not plausible, one could apply a suitable variant of the wild bootstrap (Mammen 1993); see the lecture notes of D\"umbgen (2015) for more details.

%============================
\section{Technical arguments}
\label{sec:Proofs}
%============================

\begin{proof}[\bf Proof of Lemma~\ref{Lemma:hat matrix}]
We write $\hat{\bs{\eps}} = \bs{G} \bs{\eps}$ with the companion hat matrix $\bs{G} := \bs{I} - \bs{H}$ describing the orthogonal projection on $\FF(\X)^\perp$. Since
\[
	\hat{\eps}_i - \hat{\eps}_j \ = \ \sum_{k=1}^n (G_{ik} - G_{jk}) \eps_k ,
\]
we may conclude that $\hat{\eps}_i - \hat{\eps}_j$ has a continuous distribution whenever $G_{ik} \ne G_{jk}$ for at least one index $k$, and this implies that $\Pr(\hat{\eps}_i = \hat{\eps}_j) = 0$. Otherwise,
\begin{equation}
\label{eq:Gi=Gj 1}
	G_{ik} \ = \ G_{jk} \quad\text{for} \ k=1,2,\ldots,n ,
\end{equation}
and $\Pr(\hat{\eps}_i = \hat{\eps}_j) = 1$. Since $\bs{G}^\top = \bs{G}$, condition~\eqref{eq:Gi=Gj 1} implies that
\begin{equation}
\label{eq:Gi=Gj 2}
	G_{ii} \ = \ G_{jj} \ = \ G_{ij} .
\end{equation}
On the other hand, since $\bs{G}^\top \bs{G} = \bs{G}$,
\[
	\sum_{k=1}^n (G_{ik} - G_{jk})^2 \ = \ G_{ii} + G_{jj} - 2 G_{ij} .
\]
Thus conditions~\eqref{eq:Gi=Gj 1} and \eqref{eq:Gi=Gj 2} are equivalent. Since $\bs{H} = \bs{I} - \bs{G}$, condition~\eqref{eq:Gi=Gj 2} is equivalent to
\begin{equation}
\label{eq:Gi=Gj 3}
	H_{ii} \ = \ H_{jj} \ = \ 1 + H_{ij} .
\end{equation}

Finally, denoting the columns of $\bs{H}$ with $\bs{h}_1, \bs{h}_2, \ldots, \bs{h}_n$, it follows from $\bs{H}^\top \bs{H} = \bs{H}$ and the Cauchy-Schwarz inequality that $|H_{ij}| = |\bs{h}_i^\top \bs{h}_j| \le \|\bs{h}_i\| \|\bs{h}_j\| = \sqrt{H_{ii} H_{jj}}$. Thus condition~\eqref{eq:Gi=Gj 3} implies that $H_{ii} \ge 1 - |H_{ij}| \ge 1 - \sqrt{H_{ii} H_{jj}} = 1 - H_{ii}$, whence $H_{ii} \ge 1/2$.
\end{proof}

A key ingredient for the proof of Theorem~\ref{Theorem:exact} is an elementary equality for bivariate Gaussian distributions which is well-known in the literature on robust correlation measures. For the reader's convenience we include a proof.

\begin{Lemma}
\label{Lemma:arcsin}
Let $\bs{Y}$ be a random vector with distribution $\mathcal{N}_2(\bs{0},\bs{\Sigma})$, where $\Sigma_{11}, \Sigma_{22} > 0$. Then
\[
	\Pr(Y_1 \le 0 \ \text{and} \ Y_2 \le 0)
	\ = \ \frac{\pi/2 + \arcsin(\rho)}{2\pi}
	\quad\text{with} \ \rho := \frac{\Sigma_{12}}{\sqrt{\Sigma_{11}\Sigma_{22}}} .
\]
\end{Lemma}

\begin{proof}[\bf Proof of Lemma~\ref{Lemma:arcsin}]
Since the probability in question does not change when we replace $Y_i$ with $\Sigma_{ii}^{-1/2} Y_i$, we may assume without loss of generality that $\Sigma_{11} = \Sigma_{22} = 1$ and $\Sigma_{12} = \rho$. If $\bs{Z}$ denotes a random vector with standard Gaussian distribution on $\R^2$, then $\bs{Y}$ has the same distribution as $[Z_1, \rho Z_1 + \bar{\rho} Z_2]^\top$, where $\bar{\rho} := \sqrt{1 - \rho^2}$. Now we write $\rho = \sin(\alpha)$ and $\bar{\rho} = \cos(\alpha)$ with $\alpha := \arcsin(\rho) \in [-\pi/2, \pi/2]$, and $\bs{Z} = [R\cos(\Theta), R\sin(\Theta)]^\top$, where $R := \|\bs{Z}\| > 0$ almost surely, and $\Theta$ is uniformly distributed on $[0,2\pi]$. Then
\begin{align*}
	\Pr(Y_1 \le 0 \ \text{and} \ Y_2 \le 0) \
	&= \ \Pr \bigl( \cos(\Theta) \le 0 \
		\text{and} \
		\sin(\alpha)\cos(\Theta) + \cos(\alpha)\sin(\Theta) \le 0 \bigr) \\
	&= \ \Pr \bigl( \cos(\Theta) \le 0 \
		\text{and} \ \sin(\alpha + \Theta) \le 0 \bigr) \\
	&= \ \Pr \bigl( \Theta \in [\pi/2, 3\pi/2] \
		\text{and} \ \alpha + \Theta \in [\pi, 2\pi] + 2\pi \mathbb{Z} \bigr) \\
	&= \ \Pr \bigl( \Theta \in [\pi - \alpha, 3\pi/2] \bigr) \\
	&= \ \frac{\pi/2 + \alpha}{2\pi} .
\end{align*}\\[-7ex]
\end{proof}

\begin{proof}[\bf Proof of Theorem~\ref{Theorem:exact}]
According to Lemma~\ref{Lemma:hat matrix},
\begin{align*}
	R_i \
	&= \ 1 + \sum_{k \ne i} 1_{[\eps_k \le \eps_i]}
		\ = \ 1 + \sum_{k \ne i} 1_{[\bs{a}_{ki}^\top \bs{\eps} \le 0]}
		\quad\text{and} \\
	\hat{R}_i \
	&= \ 1 + \sum_{k \ne i} 1_{[\hat{\eps}_k \le \hat{\eps}_i]}
		\ = \ 1 + \sum_{k \ne i} 1_{[\hat{\bs{a}}_{ki}^\top \bs{\eps} \le 0]}
\end{align*}
almost surely, where $\bs{a}_{ki} := \bs{e}_k - \bs{e}_i$ and $\hat{\bs{a}}_{ki} := \bs{G}(\bs{e}_k - \bs{e}_i)$ with the standard basis $\bs{e}_1, \bs{e}_2, \ldots, \bs{e}_n$ of $\R^n$. Consequently it follows from Lemma~\ref{Lemma:arcsin} that
\begin{align*}
	\Ex \bigl( & (\hat{R}_i - R_i)(\hat{R}_j - R_j) \bigr) \\
	&= \ \sum_{k \ne i, \ell \ne j} \Ex \,
		\bigl( 1_{[\bs{a}_{ki}^\top \bs{\eps} \le 0]}
			- 1_{[\hat{\bs{a}}_{ki}^\top \bs{\eps} \le 0]} \bigr)
		\bigl( 1_{[\bs{a}_{\ell j}^\top \bs{\eps} \le 0]}
			- 1_{[\hat{\bs{a}}_{\ell j}^\top \bs{\eps} \le 0]} \bigr) \\
	&= \ \sum_{k \ne i, \ell \ne j} \Bigl(
		\Pr \bigl( \bs{a}_{ki}^\top \bs{\eps} \le 0,
			\bs{a}_{\ell j}^\top \bs{\eps} \le 0 \bigr)
		+ \Pr \bigl( \hat{\bs{a}}_{ki}^\top \bs{\eps} \le 0,
			\hat{\bs{a}}_{\ell j}^\top \bs{\eps} \le 0 \bigr) \\
	& \qquad\qquad
		- \ \Pr \bigl( \bs{a}_{ki}^\top \bs{\eps} \le 0,
			\hat{\bs{a}}_{\ell j}^\top \bs{\eps} \le 0 \bigr)
		- \Pr \bigl( \hat{\bs{a}}_{ki}^\top \bs{\eps} \le 0,
			\bs{a}_{\ell j}^\top \bs{\eps} \le 0 \bigr) \Bigr) \\
	&= \ \frac{1}{2\pi} \sum_{k \ne i,\ell \ne j} \Bigl(
		\arcsin(\cos(\bs{a}_{ki}, \bs{a}_{\ell j}))
			+ \arcsin(\cos(\hat{\bs{a}}_{ki}, \hat{\bs{a}}_{\ell j})) \\
	& \qquad\qquad\qquad
		- \ \arcsin(\cos(\bs{a}_{ki}, \hat{\bs{a}}_{\ell j}))
			- \arcsin(\cos(\hat{\bs{a}}_{ki},\bs{a}_{\ell j})) \Bigr) ,
\end{align*}
where
\[
	\cos(\bs{v},\bs{w}) \ := \ \frac{\bs{v}^\top \bs{w}}{\|\bs{v}\| \|\bs{w}\|}
	\quad\text{for} \ \bs{v}, \bs{w} \in \R^n \setminus \{\bs{0}\} .
\]
Note that
\[
	\bs{a}_{ki}^\top \bs{a}_{\ell j}^{}
	\ = \ (\bs{e}_k - \bs{e}_i)^\top (\bs{e}_\ell - \bs{e}_j)
	\ = \ \delta_{k\ell} + \delta_{ij} - \delta_{kj} - \delta_{i\ell}
	\ =: \ \Delta_{k\ell,ij} ,
\]
and with
\[
	H_{k\ell,ij} \ := \ (\bs{e}_k - \bs{e}_i)^\top \bs{H} (\bs{e}_\ell - \bs{e}_j)
	\ = \ H_{k\ell} + H_{ij} - H_{kj} - H_{i\ell} ,
\]
we may write
\begin{align*}
	\bs{a}_{ki}^\top \hat{\bs{a}}_{\ell j}^{} \
	&= \ (\bs{e}_k - \bs{e}_i)^\top \bs{G} (\bs{e}_\ell - \bs{e}_j)
		\ = \ G_{k\ell} + G_{ij} - G_{kj} - G_{i\ell} \\
	&= \ \Delta_{k\ell,ij} - H_{k\ell,ij} , \\
	\hat{\bs{a}}_{ki}^\top \hat{\bs{a}}_{\ell j}^{} \
	&= \ (\bs{e}_k - \bs{e}_i)^\top \bs{G}^\top \bs{G} (\bs{e}_\ell - \bs{e}_j)
		\ = \ (\bs{e}_k - \bs{e}_i)^\top \bs{G} (\bs{e}_\ell - \bs{e}_j)\\
	&= \ \Delta_{k\ell,ij} - H_{k\ell,ij} .
\end{align*}
Hence we obtain the formula
\begin{align*}
	\Ex \bigl( & (\hat{R}_i - R_i)(\hat{R}_j - R_j) \bigr) \\
	&= \ \frac{1}{2\pi} \sum_{k \ne i, \ell \ne j}
		\biggl( \arcsin \Bigl( \frac{\Delta_{k\ell,ij}}{2} \Bigr)
		+ \arcsin \Bigl( \frac{\Delta_{k\ell,ij} - H_{k\ell,ij}}
				{\sqrt{(2 - H_{kk,ii})(2 - H_{\ell\ell,jj})}} \Bigr) \\
	& \qquad\qquad\qquad
			- \ \arcsin \Bigl( \frac{\Delta_{k\ell,ij} - H_{k\ell,ij}}
				{\sqrt{2 (2 - H_{\ell\ell,jj})}} \Bigr)
			- \arcsin \Bigl( \frac{\Delta_{k\ell,ij} - H_{k\ell,ij}}
				{\sqrt{2 (2 - H_{kk,ii})}} \Bigr) \biggr) .
\end{align*}
But the restriction to indices $k \ne i$ and $\ell \ne j$ is superfluous, because $\Delta_{k\ell,ij} = H_{k\ell,ij} = 0$ whenever $k = i$ or $\ell = j$. This yields the first asserted formula.

In the special case of $i = j$, note that $\Delta_{k\ell,ii} = 1 + \delta_{k\ell}$ if $i \not\in \{k,\ell\}$. If we replace $\Delta_{k\ell,ii}$ with $1 + \delta_{k\ell}$ in our formula for $\Ex \bigl( (\hat{R}_i - R_i)^2 \bigr)$, we end up with the expression
\begin{align*}
	\frac{1}{2\pi} \sum_{k,\ell = 1}^n
		\biggl( & \arcsin \Bigl( \frac{1 + \delta_{k\ell}}{2} \Bigr)
		+ \arcsin \Bigl( \frac{1 + \delta_{k\ell} - H_{k\ell,ii}}
				{\sqrt{(2 - H_{kk,ii})(2 - H_{\ell\ell,ii})}} \Bigr) \\
	& - \ \arcsin \Bigl( \frac{1 + \delta_{k\ell} - H_{k\ell,ii}}
				{\sqrt{2 (2 - H_{\ell\ell,ii})}} \Bigr)
			- \arcsin \Bigl( \frac{1 + \delta_{k\ell} - H_{k\ell,ii}}
				{\sqrt{2 (2 - H_{kk,ii})}} \Bigr) \biggr) .
\end{align*}
But for $k = i$ or $\ell = i$ the corresponding summands are equal to zero, because $k = i$ implies that $H_{k\ell,ii} = H_{kk,ii} = 0$, and $\ell = i$ implies that $H_{k\ell,ii} = H_{\ell\ell,ii} = 0$. Distinguishing the cases $k = \ell$ and $k \ne \ell$ yields
\begin{align*}
	\Ex \bigl( (\hat{R}_i - R_i)^2 \bigr)
	\ = \ \frac{1}{2\pi} \sum_{k = 1}^n
		\biggl(
		& \pi - 2 \arcsin \Bigl( \sqrt{1 - H_{kk,ii}/2} \Bigr) \biggr)
			\ + \\
	\frac{1}{\pi} \sum_{1 \le k < \ell \le n} \biggl(
		& \frac{\pi}{6} + \arcsin \Bigl( \frac{1 - H_{k\ell,ii}}
				{\sqrt{(2 - H_{kk,ii})(2 - H_{\ell\ell,ii})}} \Bigr) \\
		& - \ \arcsin \Bigl( \frac{1 - H_{k\ell,ii}}
			{\sqrt{2 (2 - H_{\ell\ell,ii})}} \Bigr)
		- \arcsin \Bigl( \frac{1 - H_{k\ell,ii}}
			{\sqrt{2 (2 - H_{kk,ii})}} \Bigr) \biggr) .
\end{align*}
Finally the assertion follows from the well-known fact that $\arcsin \bigl( \sqrt{1 - t} \bigr) = \arccos \bigl( \sqrt{t} \bigr)$ for $0 \le t \le 1$.
\end{proof}

\begin{proof}[\bf Proof of Theorem~\ref{Theorem:approximate}]
First recall that $|H_{k\ell}| \le \sqrt{H_{kk} H_{\ell\ell}} \le \eta$, whence $|H_{k\ell,ij}| \le 4\eta$. Furthermore, $\Delta_{k\ell,ij} = \delta_{ij}$ whenever $\{k,\ell\} \cap \{i,j\} = \emptyset$ and $k \ne \ell$, i.e.\ $\Delta_{k\ell,ij} \ne \delta_{ij}$ for at most $n+2$ index pairs $(k,\ell)$. Elementary calculus shows that
\[
	\bigl| \arcsin(x) - \arcsin(y) \bigr|
	\ \le \ C \sqrt{|x - y|}
\]
for some constant $C$, the optimal one being $\pi / \sqrt{2}$. Hence
\begin{align*}
	\biggl| & \arcsin \Bigl( \frac{d}{2} \Bigr)
		+ \arcsin \Bigl( \frac{d - H_{k\ell,ij}}
			{\sqrt{(2 - H_{kk,ii})(2 - H_{\ell\ell,jj})}} \Bigr) \\
	& - \ \arcsin \Bigl( \frac{d - H_{k\ell,ij}}
			{\sqrt{2 (2 - H_{kk,ii})}} \Bigr)
		- \arcsin \Bigl( \frac{d - H_{k\ell,ij}}
			{\sqrt{2 (2 - H_{\ell\ell,jj})}} \Bigr) \biggr|
		\ = \ O(\eta^{1/2})
\end{align*}
uniformly in $k,\ell,i,j \in \{1,2,\ldots,n\}$ and $d \in \{\Delta_{k\ell,ij}, \delta_{ij}\}$. Consequently,
\begin{align*}
	\Ex \bigl( & (\hat{R}_i - R_i)(\hat{R}_j - R_j) \bigr) \\
	&= \ \frac{1}{2\pi} \sum_{k,\ell=1}^n
		\biggl( \arcsin \Bigl( \frac{\delta_{ij}}{2} \Bigr)
		+ \arcsin \Bigl( \frac{\delta_{ij} - H_{k\ell,ij}}
			{\sqrt{(2 - H_{kk,ii})(2 - H_{\ell\ell,jj})}} \Bigr) \\
	& \qquad\qquad\qquad
		- \ \arcsin \Bigl( \frac{\delta_{ij} - H_{k\ell,ij}}
			{\sqrt{2 (2 - H_{kk,ii})}} \Bigr)
		- \arcsin \Bigl( \frac{\delta_{ij} - H_{k\ell,ij}}
			{\sqrt{2 (2 - H_{\ell\ell,jj})}} \Bigr) \biggr)
		+ O(n \eta^{1/2})
\end{align*}
uniformly in $i,j \in \{1,2,\ldots,n\}$. But for $d \in [0,1]$ and $x,y,z \in [-4\eta,4\eta]$,
\begin{align*}
	&\arcsin \Bigl( \frac{d}{2} \Bigr)
		+ \arcsin \Bigl( \frac{d - x}{\sqrt{(2 - y)(2 - z)}} \Bigr)
		- \arcsin \Bigl( \frac{d - x}{\sqrt{2(2 - y)}} \Bigr)
		- \arcsin \Bigl( \frac{d - x}{\sqrt{2(2 - z)}} \Bigr) \\
	&= \ \arcsin \Bigl( \frac{d}{2} \Bigr)
		- \arcsin \Bigl( \frac{d - x}{2} \Bigr) \\
	& \qquad + \
		\arcsin \Bigl( \frac{d - x}{2} \Bigr)
		- \arcsin \Bigl( \frac{d - x}{\sqrt{2(2 - y)}} \Bigr) \\
	& \qquad + \
		\arcsin \Bigl( \frac{d - x}{\sqrt{(2 - y)(2 - z)}} \Bigr)
		- \arcsin \Bigl( \frac{d - x}{\sqrt{2(2 - z)}} \Bigr) \\
	&= \
		\arcsin \Bigl( \frac{d}{2} \Bigr)
		- \arcsin \Bigl( \frac{d - x}{2} \Bigr) \\
	& \qquad + \
		\arcsin \Bigl( \frac{d - x}{2} \Bigr)
		- \arcsin \Bigl( \frac{d - x}{2}
			\Bigl( 1 + \frac{y}{4} + O(\eta^2) \Bigr) \Bigr) \\
	& \qquad + \
		\arcsin \Bigl( \frac{d - x}{\sqrt{2(2 - z)}}
			\Bigl( 1 + \frac{y}{4} + O(\eta^2) \Bigr) \Bigr)
		- \arcsin \Bigl( \frac{d - x}{\sqrt{2(2 - z)}} \Bigr) \\
	&= \
		\arcsin' \Bigl( \frac{d}{2} + O(\eta) \Bigr)
			\frac{x}{2} \\
	& \qquad + \
	 	\arcsin' \Bigl( \frac{d}{2} + O(\eta) \Bigr)
			\Bigl( \frac{-dy}{8} + O(\eta^2) \Bigr)
		+
		\arcsin' \Bigl( \frac{d}{2} + O(\eta) \Bigr)
			\Bigl( \frac{dy}{8} + O(\eta^2) \Bigr) \\
	&= \
		\frac{x}{\sqrt{4 - d^2}} + O(\eta^2) .
\end{align*}
Consequently,
\[
	\Ex \bigl( (\hat{R}_i - R_i)(\hat{R}_j - R_j) \bigr)
	\ = \ \frac{1}{2\pi \sqrt{4 - \delta_{ij}}}
		\sum_{k,\ell=1}^n H_{k\ell,ij} + O \bigl( n \eta^{1/2} + n^2 \eta^2 \bigr) .
\]
Finally it follows from $\bs{H} \bs{1} = \bs{H}^\top \bs{1} = \bs{1}$ that
\[
	\sum_{k,\ell=1}^n H_{k\ell,ij}
	\ = \ \sum_{k,\ell=1}^n (H_{k\ell} + H_{ij} - H_{kj} - H_{i\ell})
	\ = \ n^2 H_{ij} - n .
\]\\[-7ex]
\end{proof}

\paragraph{Acknowledgement.}
Constructive comments of Werner Stahel, Anja M\"uhlemann and Christof Str\"ahl are gratefully acknowledged.

\end{document}